\documentclass[12pt]{article}
\usepackage{geometry}
\pdfoutput=1
\usepackage{times}
\usepackage{amsmath}
\usepackage{amsthm,amscd}
\usepackage{amssymb}
\usepackage{latexsym}
\usepackage{enumerate}
\usepackage{enumitem}
\usepackage{graphics}
\usepackage{color}
\usepackage{graphicx}
\usepackage{cite}
\usepackage{tikz-cd}
\usepackage[all]{xy}
\usepackage{hyperref}
\hypersetup{
    colorlinks=true,
    urlcolor=black,
    citecolor=blue,   
    linkcolor=blue,
}

\newtheorem{theorem}{Theorem}[section]
\newtheorem{lemma}[theorem]{Lemma}

\newtheorem{question}[theorem]{Question}
\newtheorem{definition}[theorem]{Definition}

\newcommand{\brac}[1]{\left( #1\right)}

\newcommand{\sqbrac}[1]{\left[ #1\right]}

\newcommand{\mbb}[1]{\mathbb{#1}}

\DeclareMathOperator{\kernel}{kernel}

\DeclareMathOperator{\Image}{Image}

\DeclareMathOperator{\Cor}{Cor}

\DeclareMathOperator{\Sp}{Sp}
\DeclareMathOperator{\SL}{SL}
\DeclareMathOperator{\SU}{SU}
\DeclareMathOperator{\GL}{GL}
\DeclareMathOperator{\GO}{GO}
\DeclareMathOperator{\Oo}{O}
\DeclareMathOperator{\N}{N}
\DeclareMathOperator{\C}{C}
\DeclareMathOperator{\PGO}{PGO}
\DeclareMathOperator{\PGL}{PGL}
\DeclareMathOperator{\GU}{GU}
\DeclareMathOperator{\GSp}{GSp}
\DeclareMathOperator{\Spin}{Spin}
\DeclareMathOperator{\As}{\left(A,\sigma\right)}
\DeclareMathOperator{\xx}{\varkappa}
\DeclareMathOperator{\muu}{\underline{\mu}}
\DeclareMathOperator{\su}{\underline{\sigma}}
\DeclareMathOperator{\Sim}{Sim}
\DeclareMathOperator{\HH}{H}
\DeclareMathOperator{\Br}{Br}
\DeclareMathOperator{\End}{End}
\DeclareMathOperator{\Sn}{Sn(A,\sigma)}

\title{On Serre's injectivity question and norm principle} 
\author{Nivedita Bhaskhar \\ \small{Department of Mathematics \& Computer Science, Emory University, Atlanta, GA 30322, USA.}\\ \href{mailto:nbhaskh@emory.edu}{\small{\tt nbhaskh@emory.edu}}}
\date{}

\begin{document}
\maketitle

\begin{abstract}
Let $k$ be a field of characteristic not $2$. We give a positive answer to Serre's injectivity question for any smooth connected reductive $k$-group whose Dynkin diagram contains connected components only of type $A_n$, $B_n$ or $C_n$. We do this by relating Serre's question to the norm principles proved by Barquero and Merkurjev. We give a scalar obstruction defined up to spinor norms whose vanishing will imply the norm principle for the non-trialitarian $D_{n}$ case and yield a positive answer to Serre's question for connected reductive $k$-groups whose Dynkin diagrams contain components of non-trialitarian type $D_n$ also. We also investigate Serre's question for reductive $k$-groups whose derived subgroups admit quasi-split simply connected covers. 
\end{abstract}

\section{Introduction}

Let $k$ be a field. Then the following question of Serre, which is open in general, asks 

\begin{question}[Serre, \cite{SE}, p. 233]
Let $G$ be any connected linear algebraic group over a field $k$. Let $L_1,L_2,\ldots,L_r$ be finite field extensions of $k$ of degree $d_1,d_2,\ldots,d_r$ respectively such that $\gcd_i(d_i)=1$. Then is the following sequence exact ?

\[1\to \HH^1(k,G)\to \prod_{i=1}^{r}\HH^1(L_i,G)\] 
\end{question}

The classical result that the index of a central simple algebra divides the degrees of its splitting fields answers Serre's question affirmatively for the group $\PGL_n$. Springer's theorem for quadratic forms answers it affirmatively for the (albeit sometimes disconnected) group $O(q)$ and Bayer-Lenstra's theorem (\cite{BL}) for the groups of isometries of algebras with involutions. Jodi Black (\cite{JB}) answers Serre's question positively for absolutely simple simply connected and adjoint $k$-groups of classical type. In this paper, we use and extend Jodi's result to connected \textit{reductive} $k$-groups whose Dynkin diagram contains connected components only of type $A_n$, $B_n$ or $C_n$.

\begin{theorem}\label{mainthm}
Let $k$ be a field of characteristic not $2$. Let $G$ be a connected reductive $k$-group whose Dynkin diagram contains connected components only of type $A_n$, $B_n$ or $C_n$. Then Serre's question has a positive answer for $G$.
\end{theorem}

We also investigate Serre's question for reductive $k$-groups whose derived subgroups admit quasi-split simply connected covers. More precisely, we give a uniform proof for the following : 

\begin{theorem}
\label{quasisplit-serre}
Let $k$ be a field of characteristic not $2$. Let $G$ be a connected reductive $k$-group whose Dynkin diagram does not contain connected components  of type $E_8$. Assume further that its derived subgroup admits a quasi-split simply connected cover. Then Serre's question has a positive answer for $G$.
\end{theorem}

We relate Serre's question for $G$ with the norm principles of other closely related groups following a series of reductions used previously by Barquero and Merkurjev to prove the norm principles for reductive groups whose Dynkin diagrams do not contain connected components of type $D_n, E_6$ or $E_7$ (\cite{BM}). We also give a scalar obstruction defined up to spinor norms whose vanishing will imply the norm principle for the non-trialitarian $D_{n}$ case and yield a positive answer to Serre's question for connected reductive $k$-groups whose Dynkin diagrams contain components of this type also.

In the next section, we begin with preliminary reductions to restrict ourselves to the case of characteristic $0$ fields and reductive groups $G$ with $G'$ simply connected. In Section 3, we introduce the intermediate groups $\hat{G}$ and $\tilde{G}$ and relate Serre's question for $G$ to Serre's question for $\hat{G}$ and $\tilde{G}$ via the norm principle. In Section 4, we investigate the norm principle for (non-trialitarian) type $D_n$ groups and find the scalar obstruction whose vanishing will imply the norm principle for the non-trialitarian $D_{n}$ case. In the final section, we use the reduction techniques used in Sections \ref{preliminaries} and \ref{GG} to discuss Serre's question for connected reductive $k$-groups whose derived subgroups admit quasi-split simply connected covers.

\section{Preliminaries}
\label{preliminaries}
We work over the base field $k$ of characteristic not $2$ (which we show can be restricted to characteristic $0$). By a $k$-group, we mean a smooth connected linear algebraic group defined over $k$. And mostly, we will restrict ourselves to reductive groups. We say that a $k$-group $G$ satisfies $SQ$ if Serre's question has a positive answer for $G$. 

\subsection{Reduction to characteristic $0$}
Let $G$ be a connected reductive $k$-group whose Dynkin diagram contains connected components only of type $A_n$, $B_n$, $C_n$ or non-trialitarian $D_n$. Without loss of generality we may assume that $k$ is of characteristic $0$ (\cite{GI}, Pg 47). We give a sketch of the reduction argument for the sake of completeness.

Suppose that the characteristic of $k$ is $p >0$. Let $L_1,L_2,\ldots,L_r$ be finite field extensions of $k$ degree $d_1,d_2,\ldots,d_r$ respectively such that $\gcd_i(d_i)=1$ and let $\xi$ be an element in the kernel of 
\[\HH^1(k,G)\to \prod_{i=1}^{r}\HH^1(L_i,G).\] 

By a theorem of Gabber, Liu and Lorenzini (\cite{GLL}, Thm 9.2) which was pointed out to us by O. Wittenberg, we note that any torsor under a smooth group scheme $G/k$ which admits a zero-cycle of degree $1$ also admits a zero-cycle of degree $1$ whose support is \'etale over $k$. Thus without loss of generality we can assume that the given coprime extensions $L_i/k$ are in fact separable.

By (\cite{M}, Thm 1 \& 2), there exists a complete discrete valuation ring $R$ with residue field $k$ and fraction field $K$ of characteristic zero. Let $S_i$ denote corresponding \'etale extensions of $R$ with residue fields $L_i$ and fraction fields $K_i$.

There exists a smooth $R$-group scheme $\tilde{G}$ with special fiber $G$ and connected reductive generic fiber $\tilde{G}_K$. Now given any torsor $t\in \HH^1(k,G)$, there exists a torsor $\tilde{t} \in \HH^1_{et}(R, \tilde{G})$ specializing to $t$ which is unique upto isomorphism. This in turn gives a torsor $\tilde{t}_K$ in $\HH^1(K, \tilde{G}_K)$ by base change, thus defining a map $i_k : \HH^1(k, G)\to \HH^1(K, \tilde{G}_K)$ (\cite{GMS}, Pg 29). It clearly sends the trivial element to the trivial element. The map $i$ also behaves well with the natural restriction maps, i.e., it fits into the following commutative diagram :

\[\begin{tikzcd}
 \HH^1(k, G) \arrow{r}{i_k} \arrow{d} & \HH^1(K, \tilde{G}_K)\arrow{d} \\
  \prod \HH^1(L_i, G) \arrow{r}{\prod i_{L_i}} & \prod \HH^1(K_i, \tilde{G}_{K}).
\end{tikzcd}\]

Let $\tilde{\xi}$ denote the torsor in $\HH^1_{et}(R, \tilde{G})$ corresponding to $\xi$ as above. And let $i_k(\xi) = \tilde{\xi}_K$. This, therefore, is in the kernel of
\[\HH^1(K,\tilde{G}_K)\to \prod_{i=1}^{r}\HH^1(K_i,\tilde{G}_K).\]

Suppose that $\tilde{G}_K$ satisfies $SQ$. Then $\tilde{\xi}_K$ is trivial. However by (\cite{N}), the natural map $\HH^1_{et}(R, \tilde{G})\to \HH^1(K, \tilde{G}_K)$ is injective and hence $\tilde{\xi}$ is trivial in $\HH^1_{et}(R, \tilde{G})$. This implies that its specialization, $\xi$,  is trivial in $ \HH^1(k, G)$.

Thus from here on, we assume that the base field $k$ has characteristic $0$.

\subsection{Lemmata}

\begin{lemma}\label{directproduct}Let $k$-groups $G$ and $H$ satisfy $SQ$. Then $G\times_k H$ also satisfies $SQ$.
\end{lemma}
\begin{proof}
Consider the exact sequence $1\to H\to G\times_k H\xrightarrow{\pi} G\to 1$ of algebraic groups. Note that the projection map is surjective at all field points, ie, $\pi(L) : G\times_k H(L)\to G(L)$ is surjective for all fields $L/k$. Thus $1\to \HH^1(L,H)\to \HH^1(L_i,G\times_k H)$ is exact. Then a chase of the following diagram yields a proof of the lemma.

%
%

\[\begin{tikzcd}
 1\arrow{r} & \HH^1(k, H)\arrow{r} \arrow{d} & \HH^1(k,G\times_k H) \arrow{r} \arrow{d} &\HH^1(k, G) \arrow{d} \\ 
1 \arrow{r} &\prod \HH^1(L_i, H) \arrow{r} & \prod \HH^1(L_i, G\times_k H) \arrow{r}{\prod \delta_{L_i}} & \prod \HH^1(L_i, G)
\end{tikzcd}\]

\end{proof}

\begin{lemma}\label{quasitrivialext}
Let $1\to Q\to H\to G\to 1$ be a central extension of a $k$-group $G$ by a quasi-trivial torus $Q$. Then $H$ satisfies $SQ$ if and only if $G$ satisfies $SQ$.
\end{lemma}

\begin{proof} Since $Q$ is quasi-trivial, $\HH^1(L,Q)=\{1\}\ \forall \  L/k$. From the long exact sequence in cohomology, we have the following commutative diagram.

\[\begin{tikzcd}
 1\arrow{r} & \HH^1(k, H)\arrow{r} \arrow{d} & \HH^1(k,G) \arrow{r}{\delta_k} \arrow{d} & H^2(k,Q) \arrow{d} \\ 
1 \arrow{r} &\prod \HH^1(L_i, H) \arrow{r} & \prod \HH^1(L_i, G) \arrow{r}{\prod \delta_{L_i}} & \prod H^2(L_i, Q)\end{tikzcd}\]

From the above diagram, it is clear that if $G$ satisfies $SQ$, so does $H$. 

Conversely assume that $H$ satisfies $SQ$. Let $a\in \HH^1(k,G)$ become trivial in $\prod \HH^1(L_i,G)$. Then $\delta_k(a)$ becomes trivial in each $H^2(L_i,Q)$. Hence the corestriction $\Cor_{L_i/k}\left(\delta_k(a)\right)={\delta_k(a)}^{d_i}$ becomes trivial in $H^2(k,Q)$ which implies that $\delta_k(a)$ is itself trivial in $H^2(k,Q)$. Therefore $a$ comes from an element $b\in \HH^1(k,H)$ which is trivial in $\prod \HH^1(L_i,H)$. (The fact that $\HH^1(L_i,Q)=\{1\}$ guarentees $b$ is trivial in $\HH^1(L_i,H)$). Since $H$ satisfies $SQ$ by assumption, $b$ is trivial in $\HH^1(k,H)$ which implies the triviality of $a$ in $\HH^1(k,G)$.

\end{proof}

\subsection{Further reductions : z-extensions}
Recall that there is a central extension (called a $z$-extension) $1\to Q\to H\to G\to 1$ of $G$ by a quasitrivial torus $Q$ such that $H'$ is semisimple and simply connected (\cite{MS}, Prop 3.1 and \cite{BB}, Lem 1.1.4). Thus by Lemma \ref{quasitrivialext}, our given reductive group $G$ satisfies $SQ$ if it's $z$-extension $H$ does. That is,

\begin{lemma}
\label{wlog1}
Let $G_1$ be a connected reductive $k$-group such that that $G_1'=R_{E/k}(H')$ where $E/k$ is a separable field extension and $H'$ is an absolutely simple simply connected group (whose Dynkin diagram contains only connected components of classical type $A_n$, $B_n$, $C_n$ or non-trialitarian $D_n$) over $E$. If every such $G_1$ satisfies $SQ$, then so does any connected reductive $k$-group $G$ whose Dynkin diagram contains only connected components of classical type $A_n$, $B_n$, $C_n$ or non-trialitarian $D_n$.
\end{lemma}

Thus, without loss of generality, we assume from now on that $G$ is a connected reductive $k$-group with $G' = R_{E/k}(H')$ where $H'$ is one of the following groups : 

\begin{itemize}
\item[$\phantom{.}^1A_{n-1}$ :]
The special linear group $\SL_1(A)$ where $A$ is a central simple algebra of degree $n$ over $E$. 
\item[$\phantom{.}^2A_{n-1}$ :]
The special unitary group $\SU (B,\tau)$ where $B$ is a central simple algebra of degree $n$ over a quadratic extension $F$ of $E$ with a unitary involution $\tau$. 
\item[$B_n$ :]
The spinor group $\Spin (V,q)$ where $(V,q)$ is a non-degenerate quadratic space over $E$ of dimension $2n+1$.
\item[$C_n$ :]
The symplectic group $\Sp (A,\sigma)$ where $A$ is a central simple algebra of degree $2n$ over $E$ with symplectic involution $\sigma$.
\item[$D_n$ :]
(non-trialitarian) The spinor group $\Spin (A,\sigma)$ where $A$ is a central simple algebra of degree $2n$ over $E$ and $\sigma$ is an orthogonal involution.
\end{itemize}

\label{wlogGshape}

\section{Serre's question and norm principles}
\label{GG}
\subsection{Intermediate groups $\hat{G}$ and $\tilde{G}$}
\label{GGtilde}
\textit{Notations are as in Section 5 of (\cite{BM})}

In this section we work with the reductive $k$-group $G$ as assumed after Lemma \ref{wlog1} in the previous section, further assuming\footnote{Note that this condition is more restrictive than what was deduced in Lemma \ref{wlog1}. These restrictions will be removed in the Section \ref{notabssimple}} that its semisimple part $G'$ is an \textit{absolutely simple simply connected group of classical type $A_n$, $B_n$, $C_n$ or $D_n$}. Let $Z(G)=T$ and $Z(G')=\mu$.

Let $\rho : \mu\hookrightarrow S$ be an embedding of $\mu$ into a quasi-trivial torus $S$. Let $e(G',\rho)$ denote the cofibre product $\hat{G}= \frac{G'\times S}{\mu}$. We call $e(G',\rho)$ to be an \textit{envelope}  of $G'$.

\[
\begin{tikzcd}
   \mu \arrow{r}{\delta} \arrow{d}{\rho} & G' \arrow{d} \\
       S \arrow{r}{\gamma}      & \hat{G} 
\end{tikzcd}
\]

Depending on the type of $G'$, we choose envelopes $\hat{G}=e(G', \rho)$ given by the list below : 
\begin{itemize}
\item[$\phantom{,}^1 A_{n-1}$ :]
$S={G}_m$, $G'=\SL_1(A)$ and $\hat{G}=\GL_1(A)$ where $A$ is a central simple algebra of degree $n$ over $k$.
\item[$\phantom{,}^2 A_{n-1}$ :]
$S=R_{K/k}\mbb{G}_m$, $G'=\SU(B,\tau)$ and $\hat{G}=\GU (B,\tau)$ where $B$ is a central simple algebra of degree $n$ over a quadratic extension $K$ of $k$ with a unitary involution $\tau$. 
\item[$B_n$ :]
$S=\mbb{G}_m$, $G'=\Spin (V,q)$ and $\hat{G}=\Gamma^{+}(V,q)$ where $(V,q)$ is a non-degenerate quadratic space over $k$ of dimension $2n+1$.
\item[$C_n$ :]
$S=\mbb{G}_m$, $G'=\Sp (A,\sigma)$ and $\hat{G}=\GSp (A,\sigma)$ where $A$ is a central simple algebra of degree $2n$ over $k$ with symplectic involution $\sigma$.
\item[$D_n$ :]
(non-trialitarian) $S=R_{Z/k}\mbb{G}_m$, $G'=\Spin (A,\sigma)$ and $\hat{G} = \Omega(A,\sigma)$ where $A$ is a central simple algebra of degree $2n$ over $k$, $Z/k$ the discriminant quadratic extension and $\sigma$ is an orthogonal involution.
\end{itemize}

For each of the above cases, $S=Z(\hat{G})$ and $\hat{G}$ fit into an exact sequence as follows :

\[1\to S\to \hat{G} \to G'^{\ ad}\to 1.\]

Here $G'^{\ ad}$ corresponds to the adjoint group of $G'$. By the following theorem (\cite{JB}, Thm 0.2) we know that $G'^{\ ad}$ satisfies $SQ$ for $G'$ as above.
\vspace*{3mm}\\
\begin{theorem}[Jodi Black]
\label{Jodi}
Let $k$ be a field of characteristic different from 2 and let $G''$ be an absolutely simple algebraic $k$- group which is not of type $E_8$ and which is either a simply connected or adjoint classical group or a quasisplit exceptional group. Then Serre's question has a positive answer for $G''$.
\end{theorem}
\bigskip
Thus, for connected reductive groups $G$, with $G'$ absolutely simple and simply connected and for envelopes $\hat{G}$ chosen as above, Lemma \ref{quasitrivialext} implies that \textbf{the envelopes $\hat{G}$ satisfy $SQ$}.

Define an intermediate abelian group $\tilde{T}$ to be the cofibre product $\frac{T\times S}{\mu}$.

\[
\begin{tikzcd}
    \mu \arrow{r} \arrow{d}{\rho} & T \arrow{d}{\alpha} \\
        S \arrow{r}{\nu}      & \tilde{T} 
\end{tikzcd}
\]

Let the algebraic group $\tilde{G}$ be the cofibre product defined by the following diagram : 
\[\begin{tikzcd}
    G'\times T \arrow{r}{m} \arrow{d}{id\ \times \alpha} & G \arrow{d}{\beta} \\
    G'\times \tilde{T} \arrow{r}{\epsilon}     & \tilde{G}.
    \end{tikzcd}
\]

Then we have the following commutative diagram with exact rows (Prop 5.1, \cite{BM}) . Note that each row is a central extension of $\tilde{G}$. 

\[\begin{tikzcd}
 1 \arrow{r} & \mu \arrow{r}{\delta, \nu\rho} \arrow{d}{\rho}  & G'\times \tilde{T}\arrow{r}{\epsilon}\arrow{d} & \tilde{G}\arrow{r}\arrow{d}{id}  & 1 & \hspace*{5mm} (*) \\ 
1 \arrow{r} & S \arrow{r}{\gamma,\nu} & \hat{G}\times \tilde{T} \arrow{r} & \tilde{G}  \arrow{r} & 1 & \hspace*{5mm} (**)\end{tikzcd} \]

Since $\tilde{T}$ is abelian, the existence of the co-restriction map shows that $\tilde{T}$ satisfies $SQ$. Since $\hat{G}$ satisfies $SQ$, we can apply Lemmas \ref{directproduct} and \ref{quasitrivialext} to (**) to see that \textbf{$\tilde{G}$ satisfies $SQ$}.

\subsection{Norm principle and weak norm principle}

Let $f:G\to T$ be a map of $k$-groups where $T$ is an abelian $k$-group. Then we have norm maps $N_{L,k} : T(L)\to T(k)$ for any separable field extension $L/k$. 

\[\begin{tikzcd}
    G(L) \arrow{r}{f(L)}  & T(L) \arrow{d}{{N_{L/k}}} \\
    G(k) \arrow{r}{f(k)}       & T(k)
    \end{tikzcd}
    \]
    
We say that the \textit{norm principle} holds for $f:G\to T$ if for all separable field extensions $L/k$, \[N_{L/k}(\Image f(L)) \subseteq \Image f(k).\]

Note that the norm principle holds for any algebraic group homomorphism between abelian groups.

That is, we say that the \textit{norm principle} holds for $f: G\to T$ if given any separable field extension $L/k$ and any $t\in T(L)$ such that $t\in \brac{\Image f(L) : G(L)\to T(L)}$, then $N_{L/k}(t)\in \brac{\Image f(k) : G(k)\to T(k)}$.

We say that the \textit{weak norm principle} holds for $f: G\to T$ if given any $t\in T(k)$ such that $t\in \brac{\Image f(L) : G(L)\to T(L)}$, then $t^{[L:k]}=N_{L/k}(t)\in \brac{\Image f(k) : G(k)\to T(k)}$.

 It is clear that if the norm principle holds for $f$, then so does the weak norm principle. 

\begin{lemma}\label{comp1}Let $G,T,S$ be $k$-groups with $S,T$ abelian and $f:G\to T$, $h: T\hookrightarrow S$ be two $k$-group maps with $h$ injective. Then the (weak) norm principle holds for $f: G\to T$ if the (weak) norm principle holds for $h\circ f : G\to S$. 
\end{lemma}
\begin{proof} 
Let us show the statement for the norm principles. Let $t\in T(L)$ such that $f(L): G(L)\to T(L)$ maps $g\leadsto t$. Let $h(L)(t)=s\in S(L)$. Thus $h\circ f(L)(g) = s$. Since the norm principle holds for $h\circ f$, there exists a $g'\in G(k)$ so that $h\circ f(k)(g') = \N_{L/k}(s)$. 

Let $\theta = f(k)(g')\in T(k)$. Then $h: T(k)\to S(k)$ maps both $\N_{L/k}(t)$ and $\theta$ to $\N_{L/k}(s)$. As $h$ is injective, we get that $\N_{L/k}(t) = \theta\in \brac{\Image f(k): G(k)\to T(k)}$. The corresponding statement for the weak norm principles follows from a similar proof.
\end{proof}

\newpage

\subsection{Relating Serre's question and norm principle}
The deduction of SQ for $G$ from $\hat{G}$ and $\tilde{G}$ follows via the (weak) norm principles.

Let $\beta : G\to \tilde{G}$ be the embedding of $k$-groups with the cokernel $P$ isomorphic to the torus $\frac{S}{\mu}$ where $\tilde{G}$ and $G$ are as in Section \ref{GGtilde}. Thus we have the following exact sequence : \[ 1\to G\xrightarrow{\beta} \tilde{G} \xrightarrow{\pi} P\to 1.\]

\begin{lemma} If the weak norm principle holds for $\pi : \tilde{G} \to P$, then $G$ satisfies $SQ$.
\label{SQandnorm}
\end{lemma}

\begin{proof}
From the long exact sequence of cohomology, we have the following commutative diagram : 

\begin{frame}
\footnotesize
\arraycolsep=3pt
\resizebox{\linewidth}{!}{$\begin{array}{ccccccccccc}
1 & \to & G(k) & \to & \tilde{G}(k) & \xrightarrow{\pi_k} & P(k) & \xrightarrow{\delta_k} & \HH^1(k,G) & \xrightarrow{\beta_k}  & \HH^1(k, \tilde{G}) \\
 &  & \downarrow &  & \downarrow  &  & \downarrow &  & \downarrow &  & \downarrow  \\
1 & \to & \prod G(L_i) & \to & \prod \tilde{G}(L_i) & \xrightarrow{\prod \pi_{L_i}} & \prod P(L_i) & \xrightarrow{\prod \delta_{L_i}} & \prod \HH^1(L_i,G) & \to & \prod \HH^1(L_i, \tilde{G}).
\end{array}$}
\end{frame}
\phantom{.}
\vspace*{3mm}\\
Recall that $\tilde{G}$ satisfies $SQ$. Let $a\in \HH^1(k,G)$ become trivial in $\prod \HH^1(L_i, G)$. As $\tilde{G}$ satisfies $SQ$, $\beta_k(a)$ becomes trivial in $\HH^1(k,\tilde{G})$. Hence $a=\delta_k(b)$ for some $b\in P(k)$ and $\delta_{L_i}(b)$ is trivial in $\HH^1(L_i,G)$.
\vspace*{2mm}\\
 Therefore, there exist $c_i\in \tilde{G}(L_i)$ such that $\pi_{L_i}(c_i)=b$. Showing that $G$ satisfies $SQ$, ie, that $a$ is trivial, is equivalent to showing $b\in \brac{\Image \pi_k: \tilde{G}(k)\to P(k)}$. However $b\in \brac{\Image \pi_{L_i} : \tilde{G}(L_i)\to P(L_i)}$. Since the weak norm principle holds for $\pi : \tilde{G}\to P$, $b^{d_i}\in \Image \brac{\pi_k: \tilde{G}(k)\to P(k)}$ where $[L_i:k]=d_i$ for each $i$. As $\gcd_i(d_i)=1$, this means $b\in \Image \brac{\pi_{k}: \tilde{G}(k)\to P(k)}$.
\end{proof}

\subsection{Serre's question for $G$ with $G'$ not absolutely simple}
\label{notabssimple}

As assumed after Lemma \ref{wlog1}, let $G$ now be reductive with $G'=R_{E/k}(H')$ where $H'$ is an absolutely simple simply connected group of classical type over $E$. Let $H$ be an envelope listed before of $H'$ .  Observe that $R_{E/k}(H)$ is an envelope of $G'$ (\cite{BM}).

$H$ satisfies $SQ$ because it is fits into an exact sequence 
\[1\to \mathrm{quasi-trivial\ torus}\to H\to H'^{\ ad}\to 1\]
Hence $R_{E/k}(H)$ also satisfies $SQ$ because $\HH^1(L, R_{E/k}H) = \HH^1(L\otimes_k E, H)$.

The proof of Lemma \ref{SQandnorm} shows that if some envelope $\hat{G}$ satisfies $SQ$ (which in turn shows that the corresponding $\tilde{G}$ satsifies $SQ$) and the weak norm principle holds for $\tilde{G}\to P$, then $G$ satisfies $SQ$. 
Thus, using the envelope  $R_{E/k}(H)$ for $G$, we have the following :

\begin{lemma}
\label{generalSQandnorm} Let $G$ be any connected reductive $k$-group with $G'$ simply connected whose Dynkin diagram contains only connected components of classical type $A_n$, $B_n$, $C_n$ or non-trialitarian $D_n$ (as assumed after Lemma \ref{wlog1}). If the weak norm principle holds for $\pi : \tilde{G} \to P$, then $G$ satisfies $SQ$.
\end{lemma}

We recall now the norm principle of Merkurjev and Barquero for reductive groups of classical type.

\begin{theorem}[Barquero-Merkurjev, \cite{BM}]
Let $G$ be a reductive group over a field $k$. Assume that the Dynkin diagram of $G$ does not contain connected components $D_n, n \geq 4, E_6$ or $E_7$.  Let $T$ be any commutative $k$-group. Then the norm principle holds for any group homomorphism $G\to T$.
\end{theorem}

This shows that the norm principle and hence the weak norm principle holds for the map $\pi : \tilde{G}\to P$ for reductive $k$-groups $G$ as in the main theorem (Thm $\ref{mainthm}$), concluding the proof for it.

\textbf{Theorem 1.2}
\phantom{--\\}
Let $k$ be a field of characteristic not $2$. Let $G$ be a connected reductive $k$-group whose Dynkin diagram contains connected components only of type $A_n$, $B_n$ or $C_n$. Then Serre's question has a positive answer for $G$.

\section{Obstruction to norm principle for (non-trialitarian) $D_n$}

\subsection{Preliminaries}
Let $\As$ be a central simple algebra of degree $2n$ over $k$ and let $\sigma$ be an orthogonal involution. Let $\C\As$ denote its Clifford algebra which is a central simple algebra over its center, $Z/k$, the discriminant extension. Let $i$ denote the non-trivial automorphism of $Z/k$ and let $\su$ denote the canonical involution of $\C\As$.

Recall that, depending on the parity of $n$, $\su$ is either an involution of the second kind (when $n$ is odd) or of the first kind (when $n$ is even). Let $\muu :\Sim\left(\C\As, \su\right)\to R_{Z/k}\mathbb{G}_m $ denote the multiplier map sending similitude $c$ to  $\su(c)c$.

Let $\Omega\As$ be the \textit{extended Clifford group}, which is an \textit{envelope} of $\Spin\As$ as mentioned before. We recall below the map $\xx : \Omega\As(k)\to Z^*/k^*$ as defined in  (\cite{KMRT}, Pg 182).

Given $\omega\in \Omega\As(k)$, let $g\in \GO^{+}\As(k)$ be some similitude such that $\omega\leadsto gk^*$ under the natural surjection $\Omega\As(k)\to \PGO^+\As(k)$. 

Let $h = \mu(g)^{-1}g^2 \in \Oo^+\As(k)$ and let $\gamma \in \Gamma\As(k)$ be some element in the \textit{special Clifford group} which maps to $h$ under the vector representation $\chi' : \Gamma\As(k)\to \Oo^+\As(k)$. Then $\omega^2 = \gamma z$ for some $z\in Z^*$ and $\xx\left(\omega\right) = zk^*$.

Note that the map $\xx$ has $\Gamma\As(k)$ as kernel. Also if $z\in Z^*$, then $\xx(z)=z^2k^*$.

 By following the reductions in (\cite{BM}), it is easy to see that one needs to investigate whether the norm principle holds for the canonical map \[\Omega\As\to \frac{\Omega\As}{\left[\Omega\As,\Omega\As\right]}.\]

We will need to investigate the norm principle for two different maps depending on the parity of $n$.

\subsubsection*{The map $\mu_*$ for $n$ odd}
Let $U\subset \mathbb{G}_m \times R_{Z/k}\mathbb{G}_m$ be the algebraic subgroup defined by 
\[U(k) = \{(f,z)\in k^*\times Z^*| f^4 = \N_{Z/k}(z)\}.\]

Recall the map $\mu_{*} : \Omega\As\to U$ defined in (\cite{KMRT}, Pg 188) which sends 

\[\omega \leadsto \left(\muu(\omega), ai(a)^{-1}\muu(\omega)^2\right),\]

 where $\omega\in \Omega\As(k) $ and $\xx(\omega)=a \ \ k^*$. This induces the following exact sequence (\cite{KMRT}, Pg 190)
 
\[1\to \Spin\As\to \Omega\As\xrightarrow{\mu_*} U\to 1.\]

Since the semisimple part of $\Omega\As$ is $\Spin\As$, the above exact sequence 
shows that it suffices to check the norm principle for the map ${\mu}_{*}$.

\subsubsection*{The map $\muu$ for $n$ even}
Recall the following exact sequence induced by restricting $\muu$ to $\Omega\As$ (\cite{KMRT}, Pg 187) 

\[1\to \Spin\As\to \Omega\As \xrightarrow{\muu} R_{Z/k}\mathbb{G}_m\to 1.\]

Since the semisimple part of $\Omega\As$ is $\Spin\As$, the above exact sequence 
shows that it suffices to check the norm principle for the map $\muu$.

\subsection{An obstruction to being in the image of $\mu_{*}$ for $n$ odd}
\label{nodd}
Given $(f,z)\in U(k)$, we would like to formulate an obstruction which prevents $(f,z)$ from being in the image $\mu_*\left(\Omega\As(k)\right)$. Note that for $z\in Z^*$, $\mu_*(z) = (\N_{Z/k}(z), z^4)$ and hence the algebraic subgroup $U_0\subseteq U$ defined by 

\[U_0(k) = \{(N_{Z/k}(z), z^4) | z\in Z^*\}\]

has its $k$-points in the image $\mu_*\left(\Omega\As(k)\right)$. 

Let $\mu_{n[Z]}$ denote the kernel of the norm map $R_{K/k}\mu_n \xrightarrow{N} \mu_n$ where $K/k$ is a quadratic extension. Note that $\mu_{4[Z]}$ is the center of $\Spin \As$ as $n$ is odd. Also recall that (\cite{KMRT}, Prop 30.13, Pg 418)

\[\HH^1\left(k, \mu_{4[Z]}\right) \cong \frac{U(k)}{U_0(k)}.\]

Thus, we can construct the map $S : \PGO^{+}\As(k)\to \HH^1\left(k, \mu_{4[Z]}\right)$ induced by the following commutative diagram with exact rows :

\[\begin{tikzcd}
1 \arrow{r} & Z^*\arrow{r} \arrow{d}{\mu_*} & \Omega\As(k) \arrow{r}{\chi'} \arrow{d}{\mu_*} & \PGO^+\As(k)   \arrow{r} \arrow{d}{S}& 1 \\
1 \arrow{r} &  U_0(k) \arrow{r}  & U(k) \arrow{r} &  \HH^1\left(k, \mu_{4[Z]}\right) \arrow{r} & 1
\end{tikzcd}\]

The map $S$ also turns out to be the connecting map from $\PGO^+\As(k)\to \HH^1\left(k, \mu_{4[Z]}\right)$ (\cite{KMRT}, Prop 13.37, Pg 190) in the long exact sequence of cohomology corresponding to the exact sequence

\[1\to \mu_{4[Z]}\to \Spin \As \to \PGO^+\As\to 1.\]

Since the maps $\mu_* : Z^*\to U_0(k)$ and $\chi' : \Omega\As(k) \to  \PGO^+\As(k)$ are surjective, an element $(f,z)\in U(k)$ is in the image $\mu_*\left(\Omega\As(k)\right)$ if and only if its image $[f,z]\in \HH^1\left(k, \mu_{4[Z]}\right)$ is in the image $S\left(\PGO^+\As(k)\right)$.

Therefore we look for an obstruction preventing $[f,z]$ from being in the image $S(\PGO^+\As(k))$. Recall the following commutative diagram with exact rows and columns :

\[\begin{tikzcd}
 & & & 1\arrow{d} &  \\
 & &  &  \mu_2\arrow{d} &  \\
1\arrow{r} & \mu_2 \arrow{r}\arrow{d} & \Spin\As \arrow{r}{\chi}\arrow{d}{id} & \Oo^+\As \arrow{r}\arrow{d}{\pi} & 1 \\
  1 \arrow{r} & \mu_{4[Z]}\arrow{r} & \Spin \As \arrow{r}{\chi'} & \PGO^+\As\arrow{r}\arrow{d}  & 1\\
 & & & 1 &  \\
\end{tikzcd}\]

The long exact sequence of cohomology induces the following commutative diagram with exact columns (\cite{KMRT}, Prop 13.36, Pg 189)

\begin{figure}[h]
\[
\begin{tikzcd}
\Oo^+\As(k) \arrow{rr}{Sn} \arrow{d}{\pi} & & \frac{k^*}{k^{*2}} \arrow{d}{i} \\
 \PGO^+\As(k) \arrow{rr}{\mathrm{S}} \arrow{d}{\mu} & & \HH^1\left(k, \mu_{4[Z]}\right)\arrow{d}{j}\\
  \frac{k^*}{k^{*2}} &  =  & \frac{k^*}{k^{*2}} 
\end{tikzcd}\]
\caption{Spinor norms and S for $n$ odd}
\label{commdiagram}
\end{figure}
 where

\begin{itemize}
\item[]
$\mu : \PGO^+\As(k)\to \frac{k^*}{k^{*2}}$ is induced by the multiplier map $\mu : \GO^+\As\to \mathbb{G}_m$
\item[]
$i : \frac{k^*}{k^{*2}} \to \HH^1\left(k, \mu_{4[Z]}\right) = \frac{U(k)}{U_0(k)}$ is the map sending $f k^{*2}\leadsto [f, f^2]$
\item[]
$j : \frac{U(k)}{U_0(k)} = \HH^1\left(k, \mu_{4[Z]}\right)\to \frac{k^*}{k^{*2}}$ is the map sending  $[f,z]\leadsto \N(z_0)k^{*2}$ where $z_0\in Z^*$ \vspace*{1mm}\\ is such that $z_0i(z_0)^{-1} = f^{-2}z$.
\end{itemize}

\phantom{--\\}
\begin{definition}
We call an element $(f,z)\in U(k)$ to be \textit{special} if there exists a $[g]\in \PGO^{+}\As(k)$ such that $j([f,z])=\mu([g])$.
\end{definition}

Let $(f,z)\in U(k)$ be a special element and let $[g]\in \PGO^{+}\As(k)$ be such that $j([f,z])=\mu([g])$. From the discussion above, it is clear that $(f,z)$ is in the image $\mu_*\left(\Omega\As(k)\right)$ if and only if $[f,z]$ is in the image $S\left(\PGO^+\As(k)\right)$.  

Thus $S([g])[f,z]^{-1}$ is in $\kernel j = \Image i$ and hence there exists some $\alpha\in k^*$ such that 

\[[f,z] = S([g])[\alpha, \alpha^2] \in \frac{U(k)}{U_0(k)}.\]

Note that if $g$ is changed by an element in $\Oo^+\As(k)$, then $\alpha$ changes by a spinor norm by Figure 1 above. \textit{Thus given a special element, we have produced a scalar $\alpha \in k^*$ which is well defined upto spinor norms.}

\begin{align*}
[f,z]\in S\left(\PGO^+\As(k)\right) & \iff [\alpha, \alpha^2]\in S\left(\PGO^+\As(k)\right) \\
& \iff (\alpha, \alpha^2) \in \mu_*\left(\Omega\As(k)\right).
\end{align*}

This happens if and only if there exists $w\in \Omega\As(k)$ such that

\begin{align*}
\alpha &= \muu(w)\\
\alpha^2 &= \xx(w)i(\xx(w))^{-1}\muu(w)^2
\end{align*}

This implies $\xx(w)\in k^*$ and hence $w\in \Gamma\As(k)$. Thus $\alpha$ is a spinor norm, being the similarity of an element in the special Clifford group. Also note if $\alpha$ is a spinor norm, then $\alpha = \muu(\gamma)$ for some $\gamma\in \Gamma\As(k)$ and $\mu_*(\gamma) = \left(\muu(\gamma), \muu(\gamma)^2\right)$. 

Thus a special element $(f,z)$ is in the image of $\mu_*$ if and only if the produced scalar $\alpha$ is a spinor norm. We call the class of $\alpha$ in $\frac{k^*}{\Sn}$ to be the scalar obstruction preventing the \textit{special} element $(f,z)\in U(k)$ from being in the image $\mu_*\left(\Omega\As(k)\right)$.

\subsection{An obstruction to being in the image of $\muu$ for $n$ even}
\label{neven}
Given $z\in Z^*$, we would like to formulate an obstruction which prevents $z$ from being in the image $\muu\left(\Omega\As(k)\right)$ . Note that for $z\in Z^*$, $\muu(z) = z^2$ and hence the subgroup $Z^{*2}$ is in the image $\muu\left(\Omega\As(k)\right)$.

Like in the case of odd $n$, we can construct the map $S : \PGO^{+}\As(k)\to \frac{Z^*}{Z^{*2}}$ induced by the following commutative diagram with exact rows (\cite{KMRT}, Definition 13.32, Pg 187) :

\[\begin{tikzcd}
1 \arrow{r} & Z^* \arrow{r} \arrow{d}{\muu}   & \Omega\As(k) \arrow{d}{\muu} \arrow{r}{\chi'} & \PGO^+\As(k) \arrow{d}{S}  \arrow{r} & 1 \\
1 \arrow{r}  & Z^{*2}  \arrow{r} & Z^{*} \arrow{r} &   \frac{Z^*}{Z^{*2}} \arrow{r}& 1
\end{tikzcd}\]

Again by the surjectivity of the maps, $\muu: Z^*\to Z^{*2}$ and $\chi' : \Omega\As(k) \to  \PGO^+\As(k)$, an element $z\in Z^*$ is in the image $\muu\left(\Omega\As(k)\right)$ if and only if its image $[z]\in \frac{Z^*}{Z^{*2}}$ is in the image $S\left(\PGO^+\As(k)\right)$. Therefore we look for an obstruction preventing $[z]$ from being in the image $S(\PGO^+\As(k))$. And as before, we arrive at the the following commutative diagram with exact rows and columns (\cite{KMRT}, Prop 13.33, Pg 188)

\phantom{hline}

\begin{figure}[hh]
\[
\begin{tikzcd}
\Oo^+\As(k) \arrow{rr}{Sn} \arrow{d}{\pi} & & \frac{k^*}{k^{*2}} \arrow{d}{i} \\
 \PGO^+\As(k) \arrow{rr}{\mathrm{S}} \arrow{d}{\mu} & & \frac{Z^*}{Z^{*2}}\arrow{d}{j}\\
  \frac{k^*}{k^{*2}} &  =  & \frac{k^*}{k^{*2}} 
\end{tikzcd}\]

\caption{Spinor norms and S for $n$ even}
\label{commdiagram}
\end{figure}

where

\begin{itemize}
\item[]
$\mu : \PGO^+\As(k)\to \frac{k^*}{k^{*2}}$ is induced by the multiplier map $\mu : \GO^+\As\to \mathbb{G}_m$
\item[]
$i : \frac{k^*}{k^{*2}} \to \frac{Z^*}{Z^{*2}}$ is the inclusion map
\item[]
$j : \frac{Z^*}{Z^{*2}}\to \frac{k^*}{k^{*2}}$ is induced by the norm map from $Z^*\to k^*$.
\end{itemize}


\begin{definition}
We call an element $z\in Z^*$ to be \textit{special} if there exists a $[g]\in \PGO^{+}\As(k)$ such that $j([z])=\mu([g])$.
\end{definition}

Let $z\in Z^*$ be a special element and let $[g]\in \PGO^{+}\As(k)$ be such that $j([z])=\mu([g])$. As before a \textit{special} element $z\in Z^*$ is in the image $\muu\left(\Omega\As(k)\right)$ if and only if $[z]$ is in the image $S\left(\PGO^+\As(k)\right)$.

Thus $S([g])[z]^{-1}$ is in $\kernel j= \Image i$ and hence there exists some $\alpha\in k^*$ such that 

\[[z] = S([g])[\alpha] \in \frac{Z^*}{Z^{*2}}.\]

Note that if $g$ is changed by an element in $\Oo^+\As(k)$, then $\alpha$ changes by a spinor norm by Figure 2 above. \textit{Thus given a special element, we have produced a scalar $\alpha \in k^*$ which is well defined upto spinor norms.}  

\begin{align*}
[z]\in S\left(\PGO^+\As(k)\right) & \iff [\alpha]\in S\left(\PGO^+\As(k)\right) \\ 
& \iff (\alpha)\in \muu\left(\Omega\As(k)\right).
\end{align*}

Since $\alpha\in k^*$ also, this is equivalent to $\alpha$ being a spinor norm (\cite{KMRT}, Prop 13.25, Pg 184).

We call the class of $\alpha$ in $\frac{k^*}{\Sn}$ to be the scalar obstruction preventing the \textit{special} element $z\in Z^*$ from being in the image $\muu\left(\Omega\As(k)\right)$.

\subsection{Scharlau's norm principle for $\mu : \GO^+\As \to \mathbb{G}_m$}
Let $\mu : \GO^+\As\to \mathbb{G}_m$ denote the multiplier map and let $L/k$ be a separable field extension of finite degree. Let $g_1\in GO^+\As(L)$ be such that $\mu\left(g_1\right)=f_1\in L^*$. Let $f$ denote $\N_{L/k}\left(f_1\right)$. We would like to show that $f$ is in the image $\mu\left(\GO^+\As(k)\right)$. 

Note that by a generalization of Scharlau's norm principle (\cite{KMRT}, Prop 12.21; \cite{JB}, Lemma 4.3) there exists a $\tilde{g}\in \GO\As(k)$ such that $f = \mu(\tilde{g})$ . However we would like to find a \textit{proper} similitude $g\in \GO^+\As(k)$ such that $\mu(g)=f$.

We investigate the cases when the algebra $A$ is non-split and split separately.

\subsubsection*{Case I : $A$ is non-split}
Note that $g_1\in \GO^+\As(L)$. If $\tilde{g}\in \GO^+\As(k)$, we are done. Hence assume $\tilde{g}\not\in \GO^+\As(k)$. By a generalization of Dieudonn\'e's theorem (\cite{KMRT}, Thm 13.38, Pg 190), we see that the quaternion algebras

\begin{align*}
B_1 = \left(Z, f_1\right) & = 0 \in \Br(L), \\
B_2 = \left(Z, f \right) & = A \in Br(k).
\end{align*}

Since $A$ is non-split, $B_2\neq 0\in \Br(k)$. However co-restriction of $B_1$ from $L$ to $k$ gives a contradiction, because

\[0 = \Cor B_1 = \left(Z, \N_{L/k}(f_1) \right) = B_2 \in \Br(k).\]

Hence $\tilde{g}\in \GO^+\As(k)$.

\subsubsection*{Case II : $A$ is split}
Since $A$ is split, $A = \End V$ where $(V,q)$ is a quadratic space and $\sigma$ is the adjoint involution for the quadratic form $q$. Again, if $\tilde{g}\in \GO^+\As(k)$, we are done. Hence assume $\tilde{g}\not\in \GO^+\As(k)$. That is 

\[\det(\tilde{g})= -f^{2n/2} = -(f^n).\]  

Since $A$ is of even degree ($2n$) and split, there exists an isometry\footnote{Since $V$ is of even dimension $2n$, $h$ can be chosen to be a hyperplane reflection for instance} $h$ of determinant $-1$. Set $g = \tilde{g}h$. Then $\det(g)=f^n$ where $\mu(g)=f$. Thus we have found a suitable $g\in \GO^+\As(k)$ which concludes the proof of the following : 

\begin{theorem}
\label{normprincipleGO+} The norm principle holds for the map $\mu : \GO^+\As \to \mathbb{G}_m$. 
\end{theorem}

\subsection{Spinor obstruction to norm principle for non-trialitarian $D_n$}
Let $L/k$ be a separable field extension of finite degree. And let $w_1\in \Omega\As(L)$ be such that  for 

\begin{itemize}
\item[] $n$ odd : $\mu_{*}(w_1) = \theta$ which is equal to $(f_1,z_1) \in U(L)$,
\item[] $n$  even : $\muu(w_1) = \theta$ which is equal to $z_1 \in \left(R_{Z/k}\mathbb{G}_m\right)(L)$.
\end{itemize}

We would like to investigate whether $\N_{L/k}(\theta)$ is in the image of $\mu_*\left(\Omega\As(k)\right)$ (resp $\muu\left(\Omega\As(k)\right)$ ) when $n$ is odd (resp. even) in order to check if the norm principle holds for the map $\mu_* : \Omega\As\to U$ (resp. $\muu : \Omega\As\to R_{Z/k}\mathbb{G}_m$).

\newpage

Let $[g_1]\in \PGO^+\As(L)$ be the image of $w_1$ under the canonical map $\chi' : \Omega\As(L)\to \PGO^+\As(L)$. Clearly $\theta$ is \textit{special} and let $g_1\in \GO^+\As(L)$ be such that $\mu([g_1])=j([\theta])$.
\vspace*{5mm}\\
By Theorem \ref{normprincipleGO+}, there exists a $g\in \GO^+\As(k)$ such that\footnote{The map $j$ commutes with $\N_{L/k}$ in both cases.} \[\mu([g])=\N_{L/k}\left(j[\theta]\right)=j\left([\N_{L/k}\theta]\right).\] Hence $\N_{L/k}(\theta)$ is \textit{special}.

By Subsection \ref{nodd} (resp. \ref{neven}) , $\N_{L/k}(\theta)$ is in the image of $\mu_*$ (resp $\muu$) if and only if the scalar obstruction $\alpha\in \frac{k^*}{\Sn}$ defined for $\N_{L/k}(\theta)$ vanishes. Thus we have a spinor norm obstruction given below. 

\begin{theorem}[Spinor norm obstruction] Let $L/k$ be a finite separable extension of fields. Let $f$ denote the map $\mu_*$ (resp $\muu$) in the case when $n$ is odd (resp. even). Given $\theta\in f\left(\Omega\As(L)\right)$, there exists scalar obstruction $\alpha\in k^*$ such that 
\[N_{L/k}(\theta) \in f\left(\Omega\As(k)\right)\iff \alpha=1\in \frac{k*}{\Sn}.\]
\end{theorem}

Thus the norm principle for the canonical map \[\Omega\As\to \frac{\Omega\As}{\left[\Omega\As,\Omega\As\right]}\] and hence for non-trialitarian $D_n$ holds if and only if the scalar obstructions are spinor norms.

\section{Groups with quasi-split simply connected covers}
Let $G$ be a connected reductive $k$-group and let $G'$ denote its derived subgroup.  Let $G^{sc}$ denote the simply connected cover of $G'$. Then one has the exact sequence $1\to C\to G^{sc}\to G'\to 1$, where $C$ is a finite $k$-group of multiplicative type, central in $G^{sc}$. The group $C$ is also sometimes termed the \textit{fundamental group} of $G'$. 
\vspace*{5mm}\\
Let $G$ be any connected reductive $k$-group whose Dynkin diagram does not contain connected components of type $E_8$. Assume further that $G^{sc}$ is quasi-split. We would like to show that $G$ satisfies $SQ$ by following the reduction techniques used in Sections \ref{preliminaries} and \ref{GG}. 
\vspace*{3mm}\\
\begin{lemma}
\label{quasisplit} Let $G$ be a connected reductive $k$-group. If $G^{sc}$ is quasi-split, then there exists a $z$-extension $1\to Q\to H\xrightarrow{\psi} G\to 1$, where $Q$ is a quasi-trivial $k$-torus, central in reductive $k$-group $H$ with $H'$ simply connected and quasi-split.
\end{lemma}
\begin{proof}
This is simply because $\psi|_{H'}: H'\to G$ yields the simply connected cover of $G'$.
\end{proof}

Lemmata \ref{quasitrivialext} and \ref{quasisplit} imply that we can restrict ourselves to connected reductive $k$-groups $G$ such that $G'$ is simply connected and quasi-split.
\vspace*{5mm}\\
 If $k$ is a finite field, Steinberg's theorem tells us that $\HH^1(k, G)=1$. Hence $G$ satisfies $SQ$ vacaously. Therefore, let us assume that $k$ is an infinite field from here on.
\vspace*{3mm}\\
\begin{lemma} \label{sur} Let $H$ be any reductive $k$-group such that its derived subgroup $H'$ is semi-simple simply connected and quasi-split. Let $T$ denote the $k$-torus  $H/H'$. Then the natural exact sequence $1\to H'\to H\xrightarrow{\phi} T\to 1$ induces surjective maps $\phi(L) : H(L)\to T(L)$ for all field extensions $L/k$. In particular, the norm principle holds for $\phi : H\to T$. 
\end{lemma}

\begin{proof}
There exists a quasi trivial maximal torus $Q_1$ of $H'$ defined over $k$ (\cite{HS}, Lem 6.7). Let $Q_1\subset Q_2$, where $Q_2$ is a maximal torus of $H$ defined over $k$. The proof of (\cite{HS}, Lem 6.6) shows that $\phi |_{Q_2} : Q_2\to T$ is surjective and that $Q_2\cap H'$ is a maximal torus of $H'$. Since $Q_2\cap H'\subseteq Q_1$, we get the following extension of $k$-tori 
\[1\to Q_1\to Q_2\to T\to 1\]
Since $Q_1$ is quasitrivial, $\HH^1\brac{L,Q_1}=0$ for any field extension $L/k$ which gives surjectivity of $\phi(L) : Q_2(L)\to T(L)$ and hence of $\phi(L) : H(L)\to T(L)$.

\end{proof}

Let $\hat{G}$ be an envelope of $G'$ defined using an embedding of $\mu=Z(G')$ into a quasi-trivial torus $S$. Note that $G'$ is assumed to be simply connected and quasi-split and is also the derived subgroup of $\hat{G}$ by construction.

\[
\begin{tikzcd}
    \mu \arrow{r}{\delta} \arrow{d}{\rho} & G'\arrow{d} \\
        S \arrow{r}{\gamma}      & \hat{G}
\end{tikzcd}
\]

Thus, we get an exact sequence $1\to G'\to \hat{G} \to \hat{G}/G' \to 1$ to which we can apply Lemma \ref{sur} to conclude that the norm principle holds for the canonical map $\hat{G}\to \frac{\hat{G}}{\sqbrac{\hat{G},\hat{G}}}$.
\vspace*{3mm}\\
Constructing the intermediate group $\tilde{G}$ as in Section \ref{GGtilde}, we see that the norm principle also holds for the natural map  $\tilde{G}\to \tilde{G}/G$ (\cite{BM}, Prop 5.1). Then using  Thm \ref{Jodi} (\cite{JB}) and Lemma \ref{SQandnorm}, we can conclude that Thm $\ref{quasisplit-serre}$ (restated below) holds.
\vspace*{5mm}\\
\textbf{Theorem 1.3}
\phantom{--\\}
Let $k$ be a field of characteristic not $2$. Let $G$ be a connected reductive $k$-group whose Dynkin diagram does not contain connected components  of type $E_8$. Assume further that $G^{sc}$ is quasi-split. Then Serre's question has a positive answer for $G$.
\vspace*{7mm}\\

\textit{Acknowledgements} : The author acknowledges support from the NSF-FRG grant 1463882. She also thanks Professors A.S Merkurjev, R. Parimala and O.Wittenberg for their many valuable suggestions and critical comments.


\begin{thebibliography}{9999}
\bibitem[1]{BM} P. Barquero and A. Merkurjev, Norm Principle for Reductive Algebraic Groups, \textit{J. Proceedings of the International Colloquium on Algebra, Arithmetic and Geometry TIFR, Mumbai} (2000).
\bibitem[2]{BL} E. Bayer-Fluckiger and H.W. Lenstra Jr., Forms in odd degree extensions and self-dual normal bases, \textit{Amer. J. Math.} \textbf{112 (3)} (1990) : pgs 359-373.
\bibitem[3]{JB} J. Black, Zero cycles of degree one on principal homogeneous spaces , \textit{Journal of Algebra} \textbf{334} (2011) : pgs 232-246.
\bibitem[4]{BB} M. Borovoi and B. Kunyavski\u i, Formulas for the unramified Brauer group of a principal homogeneous space of a linear algebraic group, \textit{Journal of Algebra} \textbf{225(2)} (2000) : pgs 804-821.
\bibitem[5]{GLL} O. Gabber, Q. Liu, \& D. Lorenzini, The index of an algebraic variety, \textit{Inventiones mathematicae} \textbf{192(3)} (2013) : pgs 567-626.
\bibitem[6]{GMS} S. Garibaldi, A. Merkurjev, J.-P. Serre, Cohomological invariants in Galois cohomology, \textit{University Lecture Series} \textbf{28}, (2003), American Mathematical Society, Providence, RI. 
\bibitem[7]{GI} P. Gille, Serre's conjecture II: a survey, \textit{Quadratic forms, linear algebraic groups, and cohomology, Springer New York} (2010) : pgs  41-56.
\bibitem[8]{HS} D. Harari and T.Szamuely, Local-global questions for tori over $ p $-adic function fields, arXiv preprint arXiv:1307.4782 (2013).
\bibitem[9]{KMRT} M.-A. Knus, A. Merkurjev, M. Rost, and J.-P. Tignol, The book of involutions, American Mathematical Society Colloquium Publications \textbf{44} (1998), AMS.
\bibitem[10]{M} S. MacLane, Subfields and automorphism groups of $p$-adic fields, \textit{Annals of Mathematics} (1939) : 423-442.
\bibitem[11]{MS} J. S. Milne and K. Y. Shih, Conjugates of Shimura varieties, \textit{Hodge cycles, motives, and Shimura varieties, Springer Berlin Heidelberg} (1981)  : pgs 280-356.
\bibitem[12]{N} Y. Nisnevich, Rationally Trivial Principal Homogeneous Spaces and Arithmetic of Reductive Group Schemes Over Dedekind Rings, \textit{C. R. Acad. Sci. Paris}, S\'erie I, \textbf{299} , no. 1 (1984) : pgs 5–8.
\bibitem[13]{SE} J.-P. Serre,  Cohomologie galoisienne: progr\`es et probl\`emes, \textit{S\'eminaire Bourbaki} (1993/94),
\textit{Ast\'erisque} \textbf{227 (783(4))} (1995) : pgs 229-257.
\end{thebibliography}
\end{document}